\documentclass[11pt]{amsart}
\usepackage{geometry}
\usepackage{graphicx} % Required for inserting images
\usepackage{amsthm}
\usepackage{amsfonts}
\usepackage{tikz}
\usepackage{indentfirst}
\usepackage{amsmath,amssymb}

\newtheorem{thm}{Theorem}[section]
\newtheorem{lem}[thm]{Lemma}%[section]
\newtheorem{prop}[thm]{Proposition}%[section]
\newtheorem{cor}[thm]{Corollary}%[section]
\newtheorem{exa}[thm]{Example}%[section]
\newtheorem{rema}[thm]{Remark}%[section]

\DeclareMathOperator{\Hom}{Hom}%
\DeclareMathOperator{\Ext}{Ext}%
\DeclareMathOperator{\undim}{\underline\dim}

\newcommand{\cA}{\ensuremath{{\mathcal{A}}}}
\newcommand{\cB}{\ensuremath{{\mathcal{B}}}}
\newcommand{\cC}{\ensuremath{{\mathcal{C}}}}
\newcommand{\cD}{\ensuremath{{\mathcal{D}}}}
\newcommand{\cE}{\ensuremath{{\mathcal{E}}}}

\usepackage[all]{xy}
\xyoption{matrix}
\xyoption{arrow}

\newcommand{\then}{\Rightarrow}

 \newcommand{\onto}{\twoheadrightarrow}
 \newcommand{\cof}{\rightarrowtail}

\DeclareMathOperator{\End}{End}%
\newcommand{\field}[1]{\mathbb{#1}}

\newcommand{\PP}{\ensuremath{{\field{P}}}}

\newcommand{\cP}{\ensuremath{{\mathcal{P}}}}

\title{Bijection between positive clusters and projectively signed exceptional sequences}
\author{Shujian Chen and Kiyoshi Igusa}



\begin{document}

\begin{abstract}
    In 2017, Igusa and Todorov gave a bijection between signed exceptional sequences and ordered partial clusters. In this paper, we show that every term in an exceptional sequence is either relatively projective or relatively injective or both and we refine this bijection to one between projectively signed exceptional sequences and ordered partial positive clusters. We also give a characterization of relatively projective/injective objects in terms of supports of the objects in the exceptional sequence.
\end{abstract}

\maketitle

\tableofcontents

\section*{Introduction}

Exceptional sequences have become an active area of research due to several new uniform proofs for the enumeration formulas for exceptional sequences for Dynkin quivers. See \cite{M}, \cite{CD}, \cite{D}. The question of which objects in an exceptional sequence are relatively projective has become important after the result of \cite{IT13} where a bijection is given between signed exceptional sequences (where relatively projective objects are allowed to be ``negative'') and ordered clusters. This has been generalized to the $\tau$-tilting case in \cite{BuanMarsh}. See also \cite{BM} and \cite[Theorem 1.19]{CI} for a combinatorial version of this.

In \cite{IS} it was shown that, for complete exceptional sequences over $A_n$ with straight orientation, every object is either relatively projective or relatively injective. Theo Douvropoulos explained to us that this statement should be true more generally and challenged us to prove it.

In the finite type case (when $\Lambda=KQ$ for a Dynkin quiver $Q$),
this result follows from \cite{BM} and \cite{CI} using the correspondence between exceptional sequences and minimal factorizations of the Coxeter element of the Weyl group given in \cite{IT}. In \cite{BM}, it is shown that for a covering relation $v < w$ in the absolute order, either $v^{-1}w$ or $K(w)^{-1}K(v)$ is a one-way reflection, i.e., agrees with the Bruhat order, where $K(v)=v^{-1}c$, and in \cite{CI} it is shown that one-way reflections in a minimal factorization correspond to relative projective representations in the complete exceptional sequence.

Throughout this paper, $\Lambda$ will be an hereditary algebra which is finite dimensional over some field. We will prove the following theorem. The terms are defined below.% suggested to us by Theo Douvropoulos.

\begin{thm}[Theorem \ref{thm: all terms are rel proj or rel inj}]\label{A}
    Let $(E_1,\cdots,E_k,\cdots, E_n)$ be a complete exceptional sequence over $\Lambda$. Then every object $E_k$ is either relatively projective or relatively injective (or both).
\end{thm}

By \cite[Theorem 4.3]{IM}, this theorem also holds for complete exceptional collections which are unordered exceptional sequences since relative projectiveness and relative injectiveness are independent of the order of the terms $E_k$.

Based on this theorem, Theo then pointed out that assigning weight $2$ to terms that are not relatively injective will give rise to a subset of the signed exceptional sequences which is equinumerous to the positive Catalan number based on his computations.

We name this subset of signed exceptional sequences, \emph{projectively signed exceptional sequences} and we show that there's a bijection between projectively signed exceptional sequences and ordered positive partial clusters. This refines that bijection between signed exceptional sequences and ordered partial clusters by Igusa and Todorov \cite{IT13} to ordered positive partial clusters

\begin{thm}[Corollary \ref{cor: projectively signed exc seq give positive partial clusters}]
The bijection between signed exceptional sequences and ordered partial clusters sends a signed exceptional sequences $(E_k,\cdots,E_n)$ to an ordered positive partial clusters $(T_k,\cdots,T_n)$ in $mod\text-\Lambda$ if and only if the exceptional sequence is projectively signed.
\end{thm}

Furthermore, Igusa generalized the bijection between signed exceptional suquences and ordered clusters to $m$-exceptional sequences and $m$-clusters in \cite{mExcSeq}. Combining with it, we are able to prove the refinement in the generalization of $m$-exceptional sequences and $m$-clusters.

\begin{thm}[Corollary \ref{cor: projectively colored exc seq give positive m-clusters}]
The bijection between $m$-exceptional sequences of length $t$ and $t$-tuples of compatible objects of $\cC^m(\Lambda)$ sends an $m$-exceptional sequences 
\[
(E_s,\cdots,E_n)
\]
%where $s=n-t+1$ 
to a $t$-tuple 
\[
(T_s,\cdots,T_n)
\]
of compatible objects in $\bigcup_{0\le j<m} mod\text-\Lambda[j]$ if and only if the $m$-exceptional sequence contains no relatively injective terms with level $m$.
\end{thm}

Using the characterization of relatively projective and relatively injective objects in a complete exceptional sequence (given in Section \ref{ss: char of rel pro and rel inj}), we obtain a result about probability distribution of objects that are both relatively projective and relatively injective for all finite type cases.

\begin{cor}[Corollary \ref{cor: probability of object being both rel proj and rel inj}]
    For a random complete exceptional sequence, the probability that the $k$th term is both relatively projective and relatively injective is $2/h$ where $h$ is the Coxeter number of $\Lambda$.
\end{cor}

At last, we give a summary of results on how Garside element action influences relative projectiviy/injectivity.

\begin{thm}[Theorem \ref{thm E}]
For any complete exceptional sequence $E_\ast=(E_1,\cdots,E_n)$, let
\[
	\Delta(E_1,\cdots,E_n)=(E_n',E_{n-1}',\cdots,E_1')
\]
where each $E_k=\tau_k E_k$. (Also, $E_1'=E_1$.) Then
\begin{enumerate}
\item $E_k$ is relatively projective in $E_\ast$ if and only if $E_k'$ is relatively injective in $\Delta E_\ast$.
\item $E_k$ is a projective module if and only if $E_k'$ is a root.
\item $E_k$ is a root if and only if $E_k'$ is an injective module.
\item $E_k$ is relatively injective but not relatively projective if and only if $E_k'$ is relatively projective but not relatively injective.
\end{enumerate}
\end{thm}

Outline of the paper is as follows:
\begin{itemize}
    \item In Section 1, we review exceptional sequences and relative projectivity/injectivity.
    \item In Section 2, we prove that every term in an exceptional sequence is either relatively projective or relatively injective. 
    \item In Section 3, we give a characterization of an object being both relatively projective and relatively injective.
    \item In Section 4, we give an application to the probability of an object being both relatively projective and relatively injective.
    \item in Section 5, we prove the bijection between projectively signed exceptional sequences and positive clusters.
    \item In Section 6, we extend the bijection to m-exceptional sequences.
    \item In Section 7, we give a summary of results on how Garside element action influence relative projectivity/injectivity.
\end{itemize}

This paper is an updated version of the paper titled ``All terms in a complete exceptional sequence are relatively projective or relatively injective.''

\section{Exceptional sequences and relative projectivity/injectivity}

Recall that a module $E$ is called \emph{exceptional} is if it is rigid, i.e., $\Ext(E,E)=0$ and its endomorphism ring is a division algebra. In particular $E$ is indecomposable. An \emph{exceptional sequence} over a hereditary algebra $\Lambda$ is a sequence of exceptional modules $(E_1,\cdots,E_n)$ so that, for all $i<j$, $E_i\in E_j^\perp$ by which we mean
\[
	\Hom(E_j,E_i)=0=\Ext(E_j,E_i).
\]
The sequence is called \emph{complete} if it is maximal, i.e., of length $n$ where

Recall that, given any complete exceptional sequence $(E_1,\cdots,E_n)$, a term $E_k$ is called \emph{relatively projective} if it is a projective object in the perpendicular category
\[
    \cA_k:=(E_{k+1}\oplus \cdots\oplus E_n)^\perp
\]
and $E_k$ is called \emph{relatively injective} if it is an injective object of
\[
    \cB_k:=\,^\perp(E_1\oplus \cdots \oplus E_{k-1}).
\]
It is clear that $\cB_k=\,^\perp\cA_{k-1}$ and $\cA_k= \cB_{k+1}^\perp$. In particular, $\cA_k$ and $\cB_k$ depend only on $E_j$ for $j\ge k$.

We say that $E_k$ is \emph{covered} by a collection of modules if the support of $E_k$ is contained in the union of the supports of these modules.

\begin{rema}
    In this paper, we also consider exceptional sequences of length $<n$. The definition of relative projectivity and relative injectivity will be with respect to the exceptional sequence completed on the left. To emphasize this we number these shorter exceptional sequences:
    \[
        (E_k,E_{k+1},\cdots,E_n).
    \]
    The missing terms $E_1,\cdots,E_{k-1}$ should go on the left. Thus, $E_n$ is always relatively injective and it is relatively projective if and only if it is a projective module. The concepts of being relatively projective and relatively injective are independent of the choice of the missing terms $E_1,\cdots,E_{k-1}$.% because, for example, $\cA_{k-1}=(E_k\oplus\cdots\oplus E_n)^\perp$.
\end{rema}

\section{Every term is either relatively projective or relatively injective}
In this section, we will prove that every term in an exceptional sequence is either relatively projective or relatively injective. In order to prove the desired statement, we first need an equivalent definition of relatively projective/injective objects.

\begin{lem}\label{lem: B is in C perp}
Let $A\to B\to C$ be an almost split sequence over a hereditary algebra so that $C$ is exceptional. Then $B\in C^\perp$, i.e.,
\[
	\Hom(C,B)=0=\Ext(C,B).
\]
\end{lem}

\begin{proof}
This follows from the 6-term exact sequence
\[
	\Hom(C, A)\to \Hom(C, B)\to \Hom(C, C)\to \Ext(C, A)\to \Ext(C, B)\to \Ext(C, C)
\]
Since $C$ is rigid and $A=\tau C$, $\Hom(C,\tau C)=0=\Ext(C,C)$. Also, $\Hom(C, C)\to \Ext(C, A)$ is an isomorphism since it is a nonzero map between one-dimensional vector spaces over the division algebra $\End(C)$. Therefore, the remaining two terms are zero: $\Hom(C, B)=0=\Ext(C, B)$.
\end{proof}

For any complete exceptional sequence $(E_1,\cdots,E_k,\cdots,E_n)$ for a hereditary algebra $\Lambda$, we use the notation
\[
	\cA_k=(E_{k+1}\oplus \cdots \oplus E_n)^\perp
\]
to denote the wide subcategory of all $\Lambda$-modules $X$ so that
\[
	\Hom(E_j,X)=0=\Ext(E_j,X)
\]
for all $j>k$. A \emph{wide subcategory} of $mod\text-\Lambda$ is an exactly embedded abelian subcategory which is closed under extensions \cite{IT}. An object $E_k$ in the exceptional sequence is called \emph{relatively projective} if $E_k$ is a projective object of the wide subcategory $\cA_k$.

\begin{lem}\label{lem: when is Ek relatively projective} $E_k$ is relatively projective if and only if there is no epimorphism $X\onto E_k$ where $X$ is an object of $\cA_{k-1}$.
\end{lem}

\begin{proof} ($\then$) If $E_k$ is projective in $\cA_k$, then any such epimorphism would split since $\cA_{k-1}\subset \cA_k$. But $E_k\notin \cA_{k-1}$, so no such $X$ exists.

($\Leftarrow$) If $E_k$ is not projective in $\cA_k$ then there is an almost split sequence
\[
	A\to X\to E_k
\]
in $\cA_k$ where $A\neq0$. Lemma \ref{lem: B is in C perp} implies that $X\in E_k^\perp$. Thus, $X\in \cA_k\cap E_k^\perp=\cA_{k-1}$.
\end{proof}

An object $E_k$ in the complete exceptional sequence $(E_1,\cdots,E_n)$ is \emph{relatively injective} if it is an injective object of the left perpendicular category
\[
	\cB_k=\,^\perp(E_1\oplus\cdots\oplus E_{k-1})%=\,^\perp \cA_{k-1}
\]
It is well-known that $\cB_k$ is dual to $\cA_{k-1}$ in the sense that $\cA_{k-1}=\cB_k^\perp$ and $\cB_k=\,^\perp \cA_{k-1}$. Thus $\cB_{k}$ is the wide subcategory of $mod\text-\Lambda$ generated by $E_{k},\cdots,E_n$.

\begin{lem}\label{lem: when is Ek relatively injective} 
$E_k$ is relatively injective if and only if there is no monomorphism $E_k\cof Y$ where $Y\in \cB_{k+1}$.
\end{lem}

\begin{proof}
Similar to the proof of Lemma \ref{lem: when is Ek relatively projective}.
\end{proof}

We can now prove the title of this section.

\begin{thm}{A}\label{thm: all terms are rel proj or rel inj}
    Let $(E_1,\cdots,E_k,\cdots, E_n)$ be a complete exceptional sequence over $\Lambda$. Then every object $E_k$ is either relatively projective or relatively injective (or both).
\end{thm}

\begin{proof}

Suppose that $E_k$ is neither relatively projective nor relatively injective. Then, by Lemmas \ref{lem: when is Ek relatively projective} and \ref{lem: when is Ek relatively injective}, we have short exact sequences:
\[
	A\to X\to E_k
\]
\[
	E_k\to Y\to B
\]
where $X\in \cA_{k-1}=\cB_k^\perp$ and $Y\in\cB_{k+1}\subset\cB_k$. This implies in particular that $\Ext(Y,X)=0$. However, the following diagram will show that this is not true.
\[
%\xymatrixrowsep{10pt}\xymatrixcolsep{10pt}
\xymatrix{%begin xy matrix
A
\ar[d]\\
X\ar[d]\ar[r] &
	Z\ar[d]\ar[r] & B\ar[d]^=\\
E_k \ar[r]& 
	Y \ar[r]&
	B
	}%end xy matrix
\]
Since $\Ext$ is right exact, we have an epimorphism
\[
	\Ext(B,X)\onto \Ext(B,E_k)
\]
So, there is a short exact sequence $X\to Z\to B$ so that $E_k\to Y\to B$ is the pushout of this sequence along $X\onto E_k$. This gives a short exact sequence:
\[
	X\to E_k\oplus Z\to Y.
\]
But this sequence does not split since $E_k$ is not a direct summand of $X\oplus Y$ since $X\in E_k^\perp$ and $Y\in\,^\perp E_k$. So, $\Ext(Y,X)\neq0$. This contradiction proves the theorem.
\end{proof}

\section{Characterization of being both relatively projective and relatively injective}\label{ss: char of rel pro and rel inj}

In this section, we will further characterize the relatively projective/injective objects by giving a criteria for an object to be both relatively projective and relatively injective.

First, we need a lemma that gives another equivalent definition of being relatively projective.

\begin{lem}\label{lem0: Ek rel proj if doesn't ext Ek-1perp}
$E_k$ is relatively projective if and only if 
\[
	\Ext(E_k,X)=0
\]
for all $X\in \cA_k\cap E_{k-1}^\perp$.
\end{lem}

\begin{proof}
If $E_k$ is relatively projective, then $\Ext(E_k,X)=0$ for all $X\in \cA_k$. So, the conditions holds. Conversely, suppose that $E_k$ is not relatively projective. Then, we have an almost split sequence $X\to Y\to E_k$ in $\cA_k$ where $X$ is the $AR$-translate of $E_k$ in $\cA_k$. Then $X$ is the first terms of the exceptional sequence
\[
	(X,E_1,\cdots,E_{k-1},E_{k+1},\cdots,E_n).
\]
Therefore $X\in \cA_k\cap E_{k-1}^\perp$ and $\Ext(E_k,X)\neq0$.
\end{proof}

Now we can prove a result that associates being relatively projective to support of objects in the exceptional sequences.

\begin{thm}\label{thm: new characterization of rel proj}
Suppose that $E_k$ is relatively injective. Then $E_k$ is also relatively projective if and only if it is not covered by $E_j$ where $j<k$.
\end{thm}

\begin{proof}
If $E_k$ is not relatively projective then there is $X\in \cA_{k-1}$ which maps onto $E_k$. Then $X$ is covered by $E_j$ where $j<k$. Since $X$ covers $E_k$, the same holds for $E_k$.

Conversely, suppose $E_k$ is relatively projective. Then we show by downward induction on $k$ that $E_k$ is not covered by $E_j$ for $j<k$.

If $k=n$, the statement holds since $E_n$ being relatively projective implies $E_n=P_i$ is projective. Then $\Hom(P_i,E_j)=0$ implies that vertex $i$ is not in the support of any $E_j$ for $j<n$. So, these cannot cover $E_n$.

Now suppose $k<n$ and the statement holds for $k+1$. Consider the braid move $\sigma_k$ on the exceptional sequence $E_\ast=(E_1,\cdots,E_k,\cdots,E_n)$ producing
\begin{equation}\label{eq: braid move}
   	\sigma_kE_\ast=(E_1,\cdots, E_{k-1},E_{k+1},E_k',E_{k+2},\cdots,E_n). 
\end{equation}
We note that the perpendicular categories $\cA_{k+1}$ are the same in the two cases. Since $E_k$ is relatively injective, we do not have a monomorphism $E_k\hookrightarrow E_{k+1}^s$. So, there are three possible formulas for $E_k'$:
\begin{enumerate}
\item $E_{k+1}\in E_k^\perp$ making $E_k'=E_k$.
\item $\Ext(E_k,E_{k+1})=K^s\neq0$ giving an exact sequence:
\[
	0\to E_{k+1}^s\to E_k'\to E_k\to 0
\]
\item $\Hom(E_k,E_{k+1})=K^s\neq0$ and the universal morphism $E_k\to E_{k+1}^s$ is an epimorphism, not a monomorphism, giving an exact sequence:
\[
	0\to E_k'\to E_k\to E_{k+1}^s\to 0
\]
\end{enumerate}
If $E_k$ is covered by $E_j$ for $j<k$ then we see that $E_k'$ is covered by the same $E_j$ plus $E_{k+1}$. So, $E_k'$ is also covered by objects to its left in $\sigma_kE_\ast$. 
\smallskip

Claim: $E_k'$ is relatively projective in $\sigma_kE_\ast$.

\smallskip
By induction, this would give a contradiction. So, it suffices to prove this claim. In all three cases the six-term exact sequence gives:
\begin{equation}\label{eq: Ek proj iff Ek' proj}
   	\Ext(E_k,X)\cong \Ext(E_k',X) 
\end{equation}
for all $X\in E_{k+1}^\perp$. Since $E_k$ is relatively projective in $E_\ast$, $\Ext(E_k,X)=0$ for all $X\in \cA_k$. But the perpendicular category $\cA_k$ for $E_\ast$ is equal to $X^\perp\cap \cA_{k+1}$ for $\sigma_kE_\ast$. By Lemma \ref{lem0: Ek rel proj if doesn't ext Ek-1perp}, we conclude that $E_k'$ is relatively projective in $\sigma_kE_k$, proving the claim and thus the theorem.
\end{proof}

Since $E_k$ must be either relatively projective or relatively injective, we obtain the following characterization of an object being both relatively projective and relatively injective.

\begin{cor}\label{thm: covering criterion for Ek to be not rel projective}
$E_k$ is not relatively projective if and only if the following both hold.
    \begin{enumerate}
        \item $E_k$ is relatively injective and
        \item $E_k$ is covered by $E_j$ for $j<k$.
    \end{enumerate}
\end{cor}

Similarly, we have the following.
\begin{thm}\label{thm: covering criterion for Ek to be not rel injective}
$E_k$ is not relatively injective if and only if the following both hold.
    \begin{enumerate}
        \item $E_k$ is relatively projective and
        \item $E_k$ is covered by $E_j$ for $j>k$.
    \end{enumerate}
\end{thm}

This theorem, together with Theorem \ref{thm: new characterization of rel proj} gives the following which we will use in the next section to prove an application of the characterization.

\begin{cor}\label{cor: roots shift to roots}
Suppose that $E_k$ is both relatively projective and relatively injective in $E_\ast=(E_1,\cdots,E_n)$. Then, in the exceptional sequence $E_\ast'$ given in \eqref{eq: braid move}, $E_k'$ is also both relatively projective and relatively injective.
\end{cor}

\begin{proof}
Since $E_k$ is relatively injective we do not have a monomorphism $E_k\hookrightarrow E_{k+1}^s$. Therefore, we have only the three possibilities listed in the proof of Theorem \ref{thm: new characterization of rel proj}. So, we obtain Equation \ref{eq: Ek proj iff Ek' proj}. This implies $E_k'$ is relatively projective in $E_\ast'$. To see that $E_k'$ is also relatively injective, suppose not. Then, by Theorem \ref{thm: covering criterion for Ek to be not rel injective}, $E_k'$ is covered by $E_{k+2},\cdots,E_n$. Looking at the three possible cases listed in the proof of Theorem \ref{thm: new characterization of rel proj} we see that, in that case, $E_{k+1},E_{k+2},\cdots,E_n$ would cover $E_k$. By Theorem \ref{thm: covering criterion for Ek to be not rel injective}, this would imply $E_k$ is not relatively injective, giving a contradiction. Therefore $E_k'$ must also be relatively injective as claimed.
\end{proof}

\section{Application to probability of being both relatively projective and relatively injective}

Given a complete exceptional sequence $(E_1,\cdots,E_n)$ for $mod\text-\Lambda$ for $\Lambda$ hereditary, we consider the Hasse diagram of the objects $E_k$ partially ordered by inclusion of their supports. We want to characterize the roots of the Hasse diagram (the maximal elements).

\begin{thm}\label{thm: characterize rel proj-rel inj objects as roots}
    Let $(E_1,\cdots,E_n)$ be a complete exceptional sequence. Then, the following are equivalent
    \begin{enumerate}
        \item $E_k$ is both relatively projective and relatively injective.
        \item $E_k$ is a root of the Hasse diagram and is not covered by its children and the other roots.
        \item $E_k$ is not covered by the other objects in the exceptional sequence.
    \end{enumerate}
\end{thm}

\begin{proof}
Clearly (2) and (3) are equivalent. Also we know that (3) implies (1) since, if $E_k$ is not relatively projective, it is covered by $E_i$ for $i<k$ by Lemma \ref{lem: when is Ek relatively projective} and, if $E_k$ is not relatively injective then it is covered by $E_j$ for $j>k$ by Lemma \ref{lem: when is Ek relatively injective}. Therefore, it suffices to show that (1) implies (3). 

So, suppose that $E_k$ is relatively projective and relatively injective. Then, consider the mutated exceptional sequence:
\[
	E_\ast''=(E_1,\cdots,E_{k-1},E_{k+1},\cdots,E_n,Y)
\] 
Using Corollary \ref{cor: roots shift to roots} several times we get that $Y$ is relatively projective and relatively injective in $E_\ast''$. In other words, $Y$ is a projective module, say $Y=P_i$, the projective cover of the simple module $S_i$. Then none of the objects $E_j$ for $j\neq k$ have $i$ in their support. On the other hand, $i$ must be in the support of $E_k$ since the dimension vectors of the objects $E_j$ in any complete exceptional sequence span $\mathbb Z^n$. Therefore, the $E_j$ for $j\neq k$ cannot cover $E_k$. This proves (3). So, these conditions are equivalent.
\end{proof}

\begin{cor}\label{cor: characterize rel proj-rel inj objects as shifts of projective objects}
    The following are equivalent
      \begin{enumerate}
        \item $E_k$ is both relatively projective and relatively injective.
        \item There is vertex $v_i$ of the quiver of $\Lambda$ which is not in the support of any $E_j$ for $j\neq k$.
        \item There is a projective module $P_i$ so that $(E_1,\cdots,\widehat{E_k},\cdots,E_n,P_i)$ is an exceptional sequence.
                \item There is an injective module $I_i$ so that $(I_i,E_1,\cdots,\widehat{E_k},\cdots,E_n)$ is an exceptional sequence.
     \end{enumerate}  
\end{cor}

\begin{cor}\label{cor: probability of object being both rel proj and rel inj}
    For a random complete exceptional sequence, the probability that the $k$th term is both relatively projective and relatively injective is $2/h$ where $h$ is the Coxeter number of $\Lambda$.
\end{cor}

\begin{proof}
     Corollary \ref{cor: characterize rel proj-rel inj objects as shifts of projective objects} gives, for any fixed $k$, a bijection between the set of complete exceptional sequences in which $E_k$ is both relatively projective and relatively injective and the set of complete exceptional sequences in which the last term $E_n$ is projective. Ringel \cite{Ringel13} has shown that, for a random complete exceptional sequence, the probability that $E_n$ is projective is $2/h$. The result follows.
\end{proof}

{ In response to a question from the referee, we show that there is a correlation between different terms in an exceptional sequence having property rPI (relatively projective and relatively injective). In fact, the probability that two terms $E_i,E_j$ are rPI can be greater, smaller or equal to $4/h^2$. For example, in a linearly ordered quiver of type $A_n$, this probability is $3/h^2$. For linearly oriented $B_n$ or $C_n$ this probability is $4/h^2$. And, for $D_4$ with symmetric orientation (with the 3-valent vertex being either a source or a sink) the probability that two terms are rPI is $4/27>3/27=4/h^2$. The details follow.}

{\begin{cor}\label{cor: probability for k rPI}
Fix $k$ positive integers $j_1<j_2<\cdots<j_k\le n$.
    For a random exceptional sequence $(E_1,\cdots,E_n)$, the probability that $E_{j_i}$ for all $i\le k$ are relatively projective and relatively injective (rPI) is independent of the choice of indices $j_i$. Furthermore, this probability is equal to the probability that the last $k$ terms $E_{n-k+1},\cdots,E_n$ are projective.
\end{cor}}

{
\begin{proof}
    We give a bijection between the sets $\cB,\cP$ where 
    
    $\cB$ is set of all exceptional sequences in which $E_{j_1},\cdots, E_{j_k}$ are rPI and 
    
    $\cP$ is the set of exceptional sequences in which the last $k$ terms are projective. 

    Given $E_\ast$ in $\cB$, using Theorem \ref{thm: characterize rel proj-rel inj objects as roots}, let $v_i$ be the vertex in the support of $E_{j_i}$ which is not in the support of any other term in the exceptional sequence. Then the corresponding object of $\cP$ is given by removing the terms $E_{j_i}$ and adding, at the right end, the projective modules $P_{v_k},\cdots,P_{v_1}$ in that order. To see that this is an exceptional sequence we note that the other terms do not have any of the vertices $v_i$ in their support. Also, $P_{v_k}$ does not have any of the other vertices $v_1,\cdots,v_{k-1}$ in its support since the dimension vector of $P_{v_k}$ is a linear combination of the dimension vectors of $E_{j_k}$ and the terms to its right. Similar for the other terms. So, there is no homomorphism from $P_{v_i}$ to any $P_{v_j}$ for $j>i$.
    
    Conversely, take any $E_\ast$ in the set $\cP$. So, the last $k$ terms in $E_\ast$ are projective, with tops, say $v_k,\cdots,v_1$. The last term $P_{v_1}$ is the only term in the sequence having support at vertex $v_1$. When we move this to the $j_1$ position the new $E_{j_1}$ will be rPI with $v_1$ as the vertex in its support which is not in the support of any other term. Continue with the next projective $P_{v_2}$. Move this to the left to give $E_{j_2}$ with property rPI. The key point is that $j_2>j_1$. Therefore, $E_{j_2}$ is a linear combination of terms to the right of $E_{j_1}$. So, $E_{j_2}$ will not have $v_1$ in its support and $E_{j_1}$ will remain rPI. Continuing in this way we obtain $E_\ast\in\cB$. 

    These procedures are inverse to each other, so we obtain a bijection $\cB\cong \cP$.
\end{proof}
}

{
By Corollary \ref{cor: probability for k rPI}, the probability that $k$ terms are rPI is the same regardless of where the $k$ terms are. So, we can take these to be the last $k$ terms $E_{n-k+1},\cdots,E_{n}$. To be rPI, they must first be relatively projective. The following lemma characterizes those sequences in which the last $k$ terms are also relatively injective.

\begin{lem}\label{lem: last k terms rPI}
    Suppose the last $k$ terms in a complete exceptional sequence are relatively projective. Then
    \begin{enumerate}
        \item Each term $E_{n-k+i}$ has a simple top at some vertex $v_i$.
        \item $v_i$ is not in the support of any $E_j$ for $j<n-k+i$. %In particular the $v_i$ are distinct.
        \item $E_{n-k+i}$ is uniquely determined by $v_i$ since it is the projective object with top $v_i$ in the perpendicular category
        $\cA_i:=(E_{n-k+i+1}\oplus \cdots\oplus E_n)^\perp$ for $i<k$ and $\cA_k:=mod\text-\Lambda$.
        \item The number of exceptional sequences for any given sequence of vertices $v_1,\cdots,v_k$ is independent of the ordering of the sequence. 
        \item The last $k$ terms are rPI if and only if $\Hom(P_{v_i},P_{v_j})=0$ for $i<j$.
    \end{enumerate}
\end{lem}
}

{
\begin{proof} We do downward induction on $i$. For $i=k$, the statements are all true since $E_n$ is projective with simple top at some vertex $v_k$. Then $\cA_{k-1}=E_n^\perp$ is the category of representations of the quiver $Q$ with vertex $v_k$ deleted. Denote this $Q_{v_k}$. $E_{n-1}$ is relatively projective when it is a projective representation of $Q_{v_k}$. Thus $E_{n-1}$ has a simple top at a vertex $v_{k-1}$. Then $\cA_{k-1}$ is the category of representations of $Q_{v_k,v_{k-1}}$, the quiver $Q$ with vertices $v_k,v_{k-1}$ deleted. Continuing in this way, we see that (1), (2) and (3) hold.

For (4) we note that $(E_1,\cdots,E_{n-k})$ is an exceptional sequence for the quiver $Q$ with vertices $v_1,\cdots,v_k$ deleted and this is independent of the order of these vertices.

Finally, to show (5) we note first that, if $E_{n-k+1},\cdots,E_n$ are rPI then, by the bijection $\cB\cong \cP$ in the proof of Corollary \ref{cor: probability for k rPI}, $(P_{v_k},\cdots,P_{v_1})$ is an exceptional sequence. So, $\Hom(P_{v_i},P_{v_j})=0$ for $i<j$. Equivalently, there is no path in the quiver from $v_j$ to $v_i$. Conversely, suppose this second condition. Then $v_i$ is not in the support of any $E_{n-k+j}$ for $J>i$ since the support of $E_{n-k+j}$ is contained in the support of $P_{v_j}$. Therefore, by Corollary \ref{cor: characterize rel proj-rel inj objects as shifts of projective objects}, each $E_{n-k+i}$ is rPI.
\end{proof}
}

{
\begin{exa}
    For a quiver of type $A_n$, we know by \cite{Iprob} that the probability of $E_{n+1-j}$ being relative projective is $\frac {j+1}h=\frac{j+1}{n+1}$, and these events are independent. Thus
    \[=\PP(E_{n-k+1},\cdots,E_{n} \text{  relatively projective})=\frac{(k+1)!}{h^k}=\frac{(k+1)!}{(n+1)^k}.
    \]
    In the special case when $A_n$ is the linearly ordered quiver $1\to 2\to\cdots\to n$, we have $\Hom(P_{v_i},P_{v_j})=0$ if and only if $v_i<v_j$. Out of the $k!$ possible orderings, this is the only ordering of $v_i$ satisfying this condition. So, by the lemma above we get
        \[=\PP(E_{n-k+1},\cdots,E_{n} \text{  rPI})=\frac1{k!}\frac{(k+1)!}{h^k}=\frac{k+1}{h^k}.
    \]
    By Corollary \ref{cor: probability for k rPI}, $E_i$ being rPI are not independent events in this case.
\end{exa}
}

{
\begin{exa}
    Similarly, for a quiver of type $B_n$ or $C_n$, we know by \cite{IS2} that the probability of $E_{n-1-j}$ being relatively projective is $\frac{2j}{h}=\frac jn$ and these events are independent. So, in this case,
    \[=\PP(E_{n-k+1},\cdots,E_{n} \text{  relatively projective})=\frac{k!}{n^k}.
    \]
    As in the previous example, when the quiver is linearly oriented, we divide by $k!$ to get
        \[=\PP(E_{n-k+1},\cdots,E_{n} \text{ rPI})=\frac{1}{n^k}=\frac{2^k}{h^k}=\left(\frac2h\right)^k.
    \]
So, these events are independent. By Corollary \ref{cor: probability for k rPI}, $E_i$ being rPI are independent events.
\end{exa}
}

{
\begin{exa}
    Finally, we consider the quiver $D_4$ with symmetric orientation: 
\[
\xymatrixrowsep{10pt}
\xymatrixcolsep{30pt}
\xymatrix{
&& \bullet\\
& \bullet\ar[ur]\ar[dr]\ar[r]&\bullet\\
&&\bullet}
\]
$D_4$ has $162$ exception sequences and $h=6$. When the quiver has symmetric orientation, we can use Lemma \ref{lem: last k terms rPI}
to determine that there are $30$ exceptional sequences in which the last two terms are rPI. Thus:
\[
\PP(E_3,E_4 \text{ rPI})=\frac{30}{162}=\frac5{27}>\frac 3{27}=\frac4{h^2}.
\]
\end{exa}
}

\section{Bijection between positive clusters and projectively signed exceptional sequences}

We will examine the Igusa-Todorov bijection between clusters and signed exceptional sequences to show that this gives a bijection between positive clusters and projectively signed exceptional sequences where \emph{projectively signed} means terms $E_j$ are allowed to be negative only if they are not relatively injective. To match the notation above, we index incomplete exceptional sequences as $(E_k,\cdots,E_n)$ indicating that they should be completed by adding terms to the left. Thus $E_n$ is always relatively injective.

\begin{thm}[Igusa-Todorov] Over any hereditary algebra $\Lambda$ of rank $n$ and $1\le k\le n$ there is a bijection between ordered partial clusters $(T_k,T_{k+1},\cdots,T_n)$ in $mod\text-\Lambda\cup \Lambda[1]$ and signed exceptional sequences $(E_k,\cdots,E_n)$ in $mod\text-\Lambda\cup mod\text-\Lambda[1]$. Furthermore, this bijection is uniquely determined by the following linear condition.

($\ast$) $\undim T_k-\undim E_k$ is a linear combination of $\undim E_j$ for $k<j\le n$.

\noindent In particular, $T_n=E_n$.
\end{thm}

\begin{rema}\label{rem: pm exc seq have more properties} This is the special case of Theorem \ref{thm: bijection for m-clusters} when $m=1$. In fact Theorem \ref{thm: bijection for m-clusters} tells us more about signed exceptional sequences: 
\begin{enumerate}
\item[(a)] If $T_k$ is projective, then $E_k$ is relatively projective with the same sign as $T_k$.
\item[(b)] If $E_k$ is both relatively projective and relatively injective, then $T_k$ has the same sign as $E_k$. (This follows from $(\ast)$.)
\item[(c)] If $T_k,E_k$ have different signs then $E_k$ is negative and relatively projective.% but not relatively injective.
\end{enumerate}
\end{rema}

\begin{cor}\label{cor: projectively signed exc seq give positive partial clusters}
The bijection between signed exceptional sequences and ordered partial clusters sends a signed exceptional sequences $(E_k,\cdots,E_n)$ to an ordered positive partial clusters $(T_k,\cdots,T_n)$ in $mod\text-\Lambda$ if and only if the exceptional sequence is projectively signed.
\end{cor}

\begin{proof} This is a special case of Corollary \ref{cor: projectively colored exc seq give positive m-clusters} with a similar proof.
\end{proof}

%\newpage

\section{Extension to $m$-exceptional sequences}

In \cite{mExcSeq}, a bijection is given between $t$-tuples of (isomorphism classes of indecomposable) compatible objects of $\cC^m$ and $m$-exceptional sequences of length $t=n-s-1$.

Recall that, for $\Lambda$ a finite dimensional hereditary algebra over any field and $m\ge1$, the \emph{$m$-cluster category} of $\Lambda$ is defined to be
\[
	\cC^m(\Lambda):=\cD^b(mod\text-\Lambda)/\tau^{-1}[m].
\]
We represent objects of $\cC^m(\Lambda)$ by objects in the fundamental domain of $\tau^{-1}[m]$ which we take to be
\begin{equation}\label{eq: fundamental domain of Cm}
	mod\text-\Lambda[0]\cup \cdots \cup mod\text-\Lambda[m-1]\cup \Lambda[m].
\end{equation}
An \emph{$m$-cluster} is a set of $m$ (nonisomorphic) compatible indecomposable objects in this fundamental domain where $X[j], Y[k]$ are \emph{compatible} if one of the following hold
\begin{enumerate}
\item $j<k$ and $X\in Y^\perp$, i.e., $(X,Y)$ is an exceptional pair,
\item $j>k$ and $Y\in X^\perp$, i.e., $(Y,X)$ is an exceptional pair
\item $j=k$ and $X,Y$ are ext-orthogonal: $\Ext_\Lambda(X,Y)=0=\Ext_\Lambda(Y,X)$.
\end{enumerate}
Thus, $Y[k]$ is compatible with $P_i[m]$ for $k<m$ if and only if $\Hom_\Lambda(P_i,Y)=0$, i.e., $i$ is not in the support of $Y$. We consider $t$-tuples of compatible objects $(T_s,\cdots,T_n)$, where $s=n-t-1$, in the fundamental domain \eqref{eq: fundamental domain of Cm}. When $s=1$ ($t=m$) these are ordered $m$-clusters.

An \emph{$m$-exceptional sequences} is defined to be a sequence of objects $(E_s,\cdots,E_n)$ of the form $E_i=X_i[j_i]$ where $(X_s,\cdots,X_n)$ is an exceptional sequence and $0\le j_i\le m$ where the \emph{level} $j_i$ of $E_i$ is allowed to be equal to $m$ only if $X_i$ is relatively projective in the completion of the exceptional sequence to the left:
\[
	X_\ast=(X_1,\cdots,X_s,\cdots,X_n).
\]
If $X$ is a module with dimension vector $\undim X\in \mathbb Z^n$, we let
\[
    \undim X[j]:=(-1)^j\undim X\in \mathbb Z^n.
\]

The following is an extended version of Theorem 2.1.1 from \cite{mExcSeq}.

\begin{thm}\label{thm: bijection for m-clusters}
There is a bijection between the set of isomorphism classes of $m$-exceptional sequences of length $t=n-s+1$ and the set of isomorphism classes of ordered $t$-tuples of pairwise compatible objects in the $m$-cluster category for any finite dimensional hereditary algebra $\Lambda$. This bijection is uniquely determined by the following formulas for each $i\le n$:
\begin{enumerate}
    \item $\undim T_i-\undim E_i$ is an integer linear combination of $\undim T_j$ for $i<j\le n$ (and an integer linear combination of $\undim E_j$ for $i<j\le n$).
    \item If $E_i$ is in level $j$ then $T_i$ is in level $j$ or $j-1$.
\end{enumerate}
Furthermore, 
\begin{enumerate}
    \item[(a)]If $T_i=P[k]$ is a projective module with level $k$ then $E_i$ has the same level and is relatively projective in the corresponding exceptional sequence.
    \item[(b)] If $E_i$ is both relatively projective and relatively injective then it has the same level as $T_i$.
    \item[(c)]Also, if $E_i,T_i$ have different levels then $E_i$ is relatively projective (and $T_i$ is not projective by (a) and $E_i$ is not relatively injective by (b)).
\end{enumerate} 

\end{thm}

%The first two paragraphs are \cite[Thm 2.1.1]{mExcSeq}, the linear equation characterizing each $T_i$ is proved below using \cite[Prop 2.5.4]{mExcSeq}.

First we derive a corollary extending Corollary \ref{cor: projectively signed exc seq give positive partial clusters} to the $m$-cluster case. %We define a \emph{positive $m$-cluster} to be an $m$-cluster containing no objects of level $m$. This is a set of compatible objects $X[j]$ with $0\le j<m$.

\begin{cor}\label{cor: projectively colored exc seq give positive m-clusters}
The bijection between $m$-exceptional sequences of length $t$ and $t$-tuples of compatible objects of $\cC^m(\Lambda)$ sends an $m$-exceptional sequences 
\[
(E_s,\cdots,E_n)
\]
%where $s=n-t+1$ 
to a $t$-tuple 
\[
(T_s,\cdots,T_n)
\]
of compatible objects in $\bigcup_{0\le j<m} mod\text-\Lambda[j]$ if and only if the $m$-exceptional sequence contains no relatively injective terms with level $m$.
\end{cor}

\begin{proof}
Let $T_\ast=(T_s,T_{s+1},\cdots,T_n)$ be an ordered partial $m$-cluster and let $E_\ast=(E_s,\cdots,E_n)$ be the corresponding $m$-exceptional sequence. We will show, by induction on $n-s$, that $T_\ast$ has no terms in level $m$ if and only if $E_\ast$ has no relatively injective terms in level $m$. For $s=n$ this holds: $E_n$ must have level $<m$ since it is relatively injective. And $T_n=E_n$ has level $<m$.%So, the statement holds. %By induction on $n-k$, it suffices to consider $T_k$.

Suppose $s<n$ and $T_s=P_i[m]$ and $i$ is not in the support of any $T_j$ for $j>s$. This implies that $i$ is not in the support of any $E_j$ for $j>s$. But $\undim T_s$ is congruent to $\undim E_s$ modulo the span of $\undim E_j$ for $j>s$. By Lemma \ref{lem: when is Ek relatively injective}, $E_s$ is relatively injective. By Theorem \ref{thm: bijection for m-clusters}(a), $E_s$ has level $m$. So, the correspondence holds in this direction.

Conversely, suppose that $E_s$ is relatively projective and relatively injective with level $m$. By Theorem \ref{thm: covering criterion for Ek to be not rel injective}(b) $T_s$ also has level $m$. Thus $T_\ast$ has no terms in level $m$ if and only if $E_\ast$ has no relatively injective terms in level $m$.
\end{proof}

\begin{exa}
Let $\Lambda$ be a path algebra of type $A_3$ with quiver
\[
	1\rightarrow 2\leftarrow 3.
\]
The 6 indecomposable modules are the projective $P_i$ and injective $I_i$. The simple modules are $I_1, P_2, I_3$. Consider the complete exceptional sequence $(P_3,I_2,I_3)$. We take three possible ways to shift these terms and list the corresponding ordered $m$-cluster, for $m=2$, on its right:
\begin{enumerate}
\item $(P_3,I_2,I_3)\leftrightarrow (P_3,I_2,I_3)$ since these are ext-orthogonal.
\item $(P_3[2],I_2,I_3)\leftrightarrow (I_1[1],I_2,I_3)$
\item $(P_3[1],I_2[1],I_3)\leftrightarrow (P_2[1],P_1[1],I_3)$
\end{enumerate}
Property (b): In all three cases, $T_2$ has the same level as $E_2$ since $E_2=I_2[i]$ is both relatively projective and relatively injective since it is not covered by $E_1,E_3$.

Property (c): In (2) $T_1,E_1$ have different levels. So, $T_1=I_1[1]$ is not a shifted projective and $E_1=P_3[2]$ is relatively projective but not relatively injective with level one more than that of $T_1$. 

Property (a): In (1) and (3), $T_1$ is projective making $E_1$ relatively projective but not relatively injective with the same level as $T_1$. $P_3$ is not relatively injective since there is a monomorphism $P_3\hookrightarrow I_2$. (Also, the first term $E_1$ is relatively injective if and only if it is injective.) In (3), $T_2=P_1[1]$ is projective and the corresponding term $E_2=I_2[1]$ is relatively projective and relatively injective.

This example illustrates that, if $T_s$ is projective, then $E_s$ has the same level and is relatively projective, but $E_s$ might or might not be relatively injective.
\end{exa}

We first review the proof of Theorem \ref{thm: bijection for m-clusters} using the Key Lemma \cite[Lem 2.2.1]{mExcSeq} together with \cite[Prop 2.5.4]{mExcSeq}.

\begin{lem}\label{key lemma}%\cite[Lem 2.2.1]{mExcSeq}
For any $T[k]\in\cE^m(\Lambda)$ there is a bijection
\[
	\sigma_{T[k]}:\cE^m(T^\perp)\to\cE^{T[k]}.
\]
between $\cE^m(T^\perp)$ and $\cE^{T[k]}$, the set of objects in $\cE^m(\Lambda)$ compatible with $T[k]$. 
This bijection is uniquely determined by the following conditions.
\begin{enumerate}
\item $\sigma(X[i])=X[i]$ if $X[i]$ is compatible with $T[k]$.
\item Otherwise $\sigma(X[i])=Y[j]$ where $(T,Y)$ is the exceptional pair given by braid mutation of $(X,T)$ and $j=i$ or $i-1$, whichever makes
\[
	(-1)^i\undim X-(-1)^j\undim Y
\]
a multiple of $\undim T$.
\end{enumerate}
Furthermore, $A,B$ are compatible in $\cE^m(T^\perp)$ if and only if $\sigma_{T[k]}A,\sigma_{T[k]}B$ are compatible in $\cE^{T[k]}$.
\end{lem}

This lemma, proved in \cite{mExcSeq}, implies the following.

\begin{prop}\label{prop: Y projective means X rel proj and inj}
In the correspondence in Lemma \ref{key lemma} we have the following.
\begin{enumerate}
\item If $j=i-1$ then $X$ is a projective object of $T^\perp$.
\item If $Y$ is a projective $\Lambda$-module, then $i=j$ and $X$ is projective in $T^\perp$.
\end{enumerate}
\end{prop}

\begin{proof}
We take the nontrivial case $X[i]\neq Y[j]$ so that $(T,Y)$ is a braid mutation of $(X,T)$. This means we can complete both pairs using the same $n-2$ objects to make complete exceptional sequences
\[
	E_\ast=(E_1,\cdots,E_{n-2},X,T),
\]
\[
	E_\ast'=(E_1,\cdots,E_{n-2},T,Y).
\]
If $Y$ is a projective $\Lambda$ module, say $Y=P_k$, then $k$ is not in the support of the other terms. So, by Theorem \ref{thm: characterize rel proj-rel inj objects as roots}, $X$ is relatively projective and relatively injective in $E_\ast$. Thus, $X$ is a projective object of $T^\perp$ as required.

If $j=i-1$, the mutation $(X,T)\to (T,Y)$ is given by an exact sequence
\[
	0\to X\to T^p\to Y\to 0
\]
since this is the only mutation which makes $\undim X+\undim Y$ a multiple of $\undim T$. So, $X$ is not relatively injective (otherwise, the map $X\to T^p$ would split). By Theorem \ref{thm: all terms are rel proj or rel inj}, $X$ is relatively projective in $E_\ast$ and thus a projective object of $T^\perp$, as required.
\end{proof}

We can now prove Theorem \ref{thm: bijection for m-clusters} following the same outline as the proof of \cite[Thm 2.1.1]{mExcSeq}:

\begin{proof}[Proof of Theorem \ref{thm: bijection for m-clusters}]
By induction on $t=n-s+1$, we construct a bijection with the required properties:
\[
	\theta_t:\left\{
	\begin{matrix}\text{$t$-tuples of compatible}\\
	\text{objects in $\cE^m(\Lambda)$}
	\end{matrix}
	\right\}
	\to
	\left\{
	\begin{matrix}\text{$m$-exceptional sequences}\\
	\text{of length $t$ for $\Lambda$}
	\end{matrix}
	\right\}
\]
For $t=1$, $\theta_1$ is the identity map. So, assume $t\ge2$.

Let $(T_s,\cdots,T_n)$ be a $t$-tuple of compatible objects in $\cE^m(\Lambda)$ and let $T_n=T[k]$ be the last object. Then $T_i$ for $i<n$ are compatible objects of $\cE^{T[k]}$. Using the inverse of the bijection 
\[
	\sigma_{T[k]}:\cE^m(T^\perp)\cong \cE^{T[k]}
\]
from Lemma \ref{key lemma}, we get
\[
	\sigma_{T[k]}^{-1}(T_s,\cdots,T_{n-1})=(X_s,\cdots,X_{n-1})
\]
a $(t-1)$-tuple of compatible elements of $\cE^m(T^\perp)$. By induction on $t$ we have a bijection:
\[
	\theta_{t-1}:\left\{
	\begin{matrix}\text{$(t-1)$-tuples of compatible}\\
	\text{objects in $\cE^m(T^\perp)$}
	\end{matrix}
	\right\}
	\to
	\left\{
	\begin{matrix}\text{$m$-exceptional sequences}\\
	\text{of length $t-1$ for $T^\perp$}
	\end{matrix}
	\right\}
\]
Let
\[
	\theta_{t-1}(X_s,\cdots,X_{n-1})=(E_s,\cdots,E_{n-1}).
\]
Then $(E_s,\cdots,E_{n-1},E_n)$, with $E_n=T_n=T[k]$ is an $m$-exceptional sequence for $\Lambda$ and we define this to be equal to $\theta_t(T_s,\cdots,T_n)$. The inverse of $\theta_t$ is given in the obvious way. (See the proof of \cite[Thm 2.1.1]{mExcSeq}.)

Now we verify conditions (a), (b), (c). For (a), suppose that $T_s=P[j]$ where $P$ is a projective $\Lambda$-module. Then, by Proposition \ref{prop: Y projective means X rel proj and inj}, $X_s=X[j]$ where $X$ is a projective object of $T^\perp$. By induction on $t$, this implies that $E_s$ is a relatively projective object with the same level $j$ in the exceptional sequence $(E_s,\cdots,E_{n-1})$ in $T^\perp$ and thus it is relatively projective in $(E_s,\cdots,E_n)$.

For (c), suppose that $T_s$ and $E_s$ have different levels. Then either 
\begin{enumerate}
\item $X_s$, $T_s$ have different levels or 
\item $X_s$, $T_s$ have the same level and $E_s$ has a different level. 
\end{enumerate}
In the second case, we have by induction that $E_s$ is relatively projective with level one more than the level of $X_s$ which is the same as the level of $T_s$. So, the required condition holds in Case (2).

In Case (1), $X_s$ is a (shifted) projective object of $T^\perp$. By induction on $t$, this makes $E_s$ relatively projective with the same level as $X_s$. Thus $E_s$ is relatively projective of level one more than that of $T_s$. 

So, this condition holds in both cases. The linear conditions (1) and (2) hold by construction. And they imply (b): If $E_s$ is both relatively projective and relatively injective, then it has a vertex in its support which is not in the support of any other $E_i$. In order for $\undim T_s$ to be congruent to $\undim E_s$ modulo the other terms, $T_s,E_s$ must have the same sign and thus the same level.

All conditions are verified and the theorem follows.
\end{proof}

\section{Garside element action and relative projectivity/injectivity}

Let $\Delta$ denote the Garside element of the braid group. This is given by
\[
	\Delta=\delta_n \delta_{n-1} \cdots \delta_2
\]
where $\delta_k=\sigma_1\sigma_2\cdots \sigma_{k-1}$. On complete exceptional sequences $\delta_k$ is given by deleting the term $E_k$ and inserting at the beginning $E_k'=\tau_k E_k$ where $\tau_k$ is Auslander-Reiten translation in the category $\cA_k=(E_{k+1}\oplus\cdots E_n)^\perp$:
\begin{equation}\label{eq: delta k of East}
	\delta_k(E_1,\cdots,E_k,\cdots,E_n)=(E_k',E_1,\cdots,\widehat{E_k},\cdots,E_n).
\end{equation}
See \cite{IS} for more details about the action of $\delta_k$ and $\Delta$ on complete exceptional sequences.

In \cite{CI} we pointed out that the ``reverse Garside'' $D\Delta$ takes relative projectives to relative projective. This is equivalent to statement (1) in the theorem below. In \cite{IS} we showed in some cases that $\Delta$ takes projective modules to ``roots'' and roots to injective modules in an exceptional sequence (statements (2), (3) in the theorem below). Extending this pattern to other terms in an exceptional sequence gives a good summary of some of our results. In this statement we define $E_k$ to be a \emph{root} if and only it is both relatively projective and relatively injective. (See Theorem \ref{thm: characterize rel proj-rel inj objects as roots}.)

\begin{thm}\label{thm E}
For any complete exceptional sequence $E_\ast=(E_1,\cdots,E_n)$, let
\[
	\Delta(E_1,\cdots,E_n)=(E_n',E_{n-1}',\cdots,E_1')
\]
where each $E_k=\tau_k E_k$. (Also, $E_1'=E_1$.) Then
\begin{enumerate}
\item $E_k$ is relatively projective in $E_\ast$ if and only if $E_k'$ is relatively injective in $\Delta E_\ast$.
\item $E_k$ is a projective module if and only if $E_k'$ is a root.
\item $E_k$ is a root if and only if $E_k'$ is an injective module.
\item $E_k$ is relatively injective but not relatively projective if and only if $E_k'$ is relatively projective but not relatively injective.
\end{enumerate}
\end{thm}

\begin{proof}
The key point is that $\cA_k$ is equal to
\[
	\cB_k':=\,^\perp(E_n'\oplus\cdots\oplus E_{k+1}').
\]
Thus (1) $E_k$ is relatively projective in $E_\ast$ if and only if $E_k$ is a projective object of $\cA_k=\cB_k'$ if and only if $E_k'=\tau_kE_k$ is an injective object of $\cA_k=\cB_k'$ if and only if $E_k'$ is relatively injective in $\Delta E_\ast$.

To show (2) we examine the exceptional sequences
\begin{equation}\label{eq: first step of Delta}
	(E_n',\cdots,E_{k+1}',E_1,\cdots,E_k)
\end{equation}
\begin{equation}\label{eq: second step of Delta}
	(E_n',\cdots,E_{k+1}',E_k', E_1,\cdots,E_{k-1})
\end{equation}
These are exceptional sequences since $\cA_k=\cB_k'$. By Corollary \ref{cor: characterize rel proj-rel inj objects as shifts of projective objects} (3), $E_k$ is projective if and only if $E_k'$ is a root in \eqref{eq: second step of Delta} if and only if $E_k'$ is a root in $\Delta E_\ast$.

(3) follows from Equation \eqref{eq: delta k of East} which shows, by Corollary \ref{cor: characterize rel proj-rel inj objects as shifts of projective objects} (4), that $E_k'$ is an injective module if and only if $E_k$ is a root in $E_\ast$.

(4) is equivalent to (1) by Theorem \ref{thm: all terms are rel proj or rel inj}.
\end{proof}

\section*{Acknowledgments}

The authors thank Theo Douvropoulos for many discussion about this problem. The second author thanks the Simons Foundation for its support: Grant \#686616.


\begin{thebibliography}{aa}
\bibitem{BM} Biane, Philippe, and Matthieu Josuat-Verg\`es. \emph{Noncrossing partitions, Bruhat order and the cluster complex.} Annales de l'Institut Fourier. Vol. 69. No. 5. 2019. %cited 3 times

\bibitem{BuanMarsh} Buan, Aslak Bakke, and Bethany Rose Marsh. \emph{$\tau$-exceptional sequences.} Journal of Algebra 585 (2021): 36--68. %cited once

\bibitem{CD} Chapuy, Guillaume, and Theo Douvropoulos. \emph{Counting chains in the noncrossing partition lattice via the W-Laplacian.} Journal of Algebra 602 (2022): 381--404.%cited once



\bibitem{CI}  Chen, Shujian and Kiyoshi Igusa, \emph{Generalized Goulden-Yong duals and signed minimal factorizations.} arXiv:2311.05732. %[Theorem 1.19]{CI} %cited 4 times


\bibitem{D} Douvropoulos, Theo. \emph{On enumerating factorizations in reflection groups.} Algebr. Comb. 6.2 (2023), pp. 359--385. issn: 2589--5486. %cited once

\bibitem{Iprob} Igusa, Kiyoshi. \emph{Probability distribution for exceptional sequences of type $ A_n$.} arXiv:2112.04996. %cited once

\bibitem{IM} Igusa, Kiyoshi, and Ray Maresca. \emph{On Clusters and Exceptional Sets in Types $\mathbb A$ and $\tilde {\mathbb A}$.} arXiv:2209.12326. %cited once

\bibitem{IS} Igusa, Kiyoshi, and Emre Sen, \textit{Exceptional sequences and rooted labeled forests}. arXiv:2108.11351. %cited 3 times

\bibitem{IS2} Igusa, Kiyoshi, and Emre Sen. \emph{Rooted labeled trees and exceptional sequences of type $ B_n/C_n$.} arXiv:2310.01700.%cited once

\bibitem{IT13} Igusa, Kiyoshi, and Gordana Todorov, \textit{Signed exceptional sequences and the cluster morphism category}, arXiv: 1706.2222.% cited 2 times

\bibitem{mExcSeq} Igusa, Kiyoshi, \emph{Enumerating $m$-clusters using exceptional sequences}, arXiv:2402.12494. % cited 11 times

\bibitem{IT} Ingalls, Colin, and Hugh Thomas. \emph{Noncrossing partitions and representations of quivers.} Compositio Mathematica 145.6 (2009): 1533--1562. %cited twice

\bibitem{M} Michel, Jean. \emph{Deligne-Lusztig theoretic derivation for Weyl groups of the number of reflection factorizations of a Coxeter element.} Proc of Amer Math Soc 144.3 (2016): 937--941. %cited once

\bibitem{Ringel13} Obaid, A.A., S. K. Nauman, W. S. Al Shammakh, W. M. Fakieh and C. M. Ringel: \emph{The
number of complete exceptional sequences for a Dynkin algebra}, Colloq. Math. 133 (2013), 197--210. % cited once

\end{thebibliography}
\end{document}